\documentclass[a4paper,reqno,12pt]{amsart}

\usepackage[utf8]{inputenc}
\usepackage{amsmath, amssymb, amsthm, epsfig,mathrsfs}
\usepackage{hyperref, latexsym}
\usepackage{url}
\usepackage{color}
\usepackage{harpoon}
\usepackage{mathabx}
\usepackage{tikz}

\usepackage{fullpage} 
\usepackage{setspace}
\usepackage{adjustbox}

\usepackage{mathtools}
\mathtoolsset{showonlyrefs}

\def\today{\ifcase\month\or
  January\or February\or March\or April\or May\or June\or
  July\or August\or September\or October\or November\or December\fi
  \space\number\day, \number\year}

\newtheorem{theorem}{Theorem}

\newtheorem{lemma}[theorem]{Lemma}
\newtheorem{proposition}[theorem]{Proposition}
\newtheorem{remark}[theorem]{Remark}

\newcommand{\C}{\mathscr{C}}

\renewcommand{\P}{\Phi}

\renewcommand{\S}{\mathcal{S}}

\newcommand{\z}{\mathbb{Z}}

\renewcommand{\r}{\mathbb{R}}
\newcommand{\cp}{\mathbb{C}} 

\newcommand{\ft}{\widehat}

\newcommand{\ga}{\gamma}

\newcommand{\ep}{\varepsilon}

\newcommand{\p}{\varphi}
\renewcommand{\d}{{\rm d}}
\newcommand{\ov}{\widebar}
\newcommand{\sign}{{\rm sign}}
\renewcommand{\sp}{\mathbb{S}}

\begin{document}


\title[]{The Hörmander--Bernhardsson function in higher dimensions}
\author[Gon\c{c}alves, Radchenko, Ramos]{Felipe Gon\c{c}alves, Danylo Radchenko, and Antonio Pedro Ramos}
\date{\today}
\subjclass[2020]{30D15, 33E30, 34A30, 41A44, 42A05}
\keywords{}
\address{IMPA - Instituto de Matemática Pura e Aplicada, Rio de Janeiro, 22460-320, Brazil.}
\email{goncalves@impa.br}

\address{Institut des Hautes \'Etudes Scientifiques, CNRS, Laboratoire Alexander Grothendieck,
35 route de Chartres, Bures-sur-Yvette 91440, France}
\email{danradchenko@gmail.com}

\address{IMPA - Instituto de Matemática Pura e Aplicada, Rio de Janeiro, 22460-320, Brazil.}
\email{antonio.ramos@impa.br}
\allowdisplaybreaks


\begin{abstract}
We study the problem of finding the norm of the point evaluation operator in the Paley--Wiener space $PW^{1}(\r^d)$, consisting of $d$-variable functions of spherical exponential type that are integrable on $\r^d$. The extremal functions can be taken radial, which naturally leads us to consider a related extremal problem in a weighted Paley--Wiener space of single-variable functions. We establish that the radial extremal function must satisfy a third-order linear ODE with polynomial coefficients for every $d \geq 1$, extending Gorbachev's recent odd-dimensional result. Along the way, we prove interpolation and reciprocal formulas involving the zeros of the extremizer. 
\end{abstract}


\maketitle

\section{Introduction} 
Let $PW^p(\r^d)$ be the space of functions $F\in L^p(\r^d)$ such that their distributional Fourier transform $\ft F$ is supported in the Euclidean unit ball. We use the following unitary normalization for the Fourier transform:
$$
\ft F(\xi) = (2\pi)^{-d/2} \int_{\r^d} F(x)e^{-i \xi \cdot x}\d x.
$$
We are interested in the norm of the point evaluation operator $F \in PW^1(\r^d) \mapsto F(0)$,
which by translation invariance can be equivalently thought of as the norm of the embedding $PW^1 (\r^d)\hookrightarrow PW^\infty (\r^d) $. That is, we seek to study the extremal problem
\begin{equation*}\label{def:const_def}
\C_d = \frac{|\sp^{d-1}|}{2} \sup_{F \in PW^1(\r^d)} \frac{\|F\|_{\infty}}{\|F\|_{1}},
\end{equation*}
where $\|\cdot\|_p:=\|\cdot\|_{L^p(\r^d)}$.
The normalization constant in the definition above is natural: this problem has a unique radial extremizer, which is the lift of a single-variable entire function belonging to a weighted Paley--Wiener space. More explicitly, after a series of routine reductions the problem can be reformulated as 
\begin{align}\label{def:hormconstant2}
\frac{1}{\C_d}=  \inf_{f\in PW_{d-1} \, : \,  f(0)=1} \|f\|_{|\cdot|^{d-1}},
\end{align}
where we denote by $PW_{d-1}$ the space of functions $f$ such that $\text{supp}(\ft f)\subset [-1,1]$ and $\| f\|_{|\cdot|^{d-1}} := \int_\r |f(x)||x|^{d-1}\d x<\infty$.

The one-dimensional version of this problem has gathered interest throughout the years, featuring, for example, in \cite{AKP,BCOS,CMS,Gor1,HB, LL, Lu}. The recent work of Bondarenko, Ortega-Cerdà, Radchenko, and Seip \cite{BORS1,BORS2} considerably improved the existing estimates for $\C_1$, as well as uncovered several properties of the extremal function. Different generalizations have also been studied: Dai, Gorbachev, and Tikhonov \cite{DGT} and Gorbachev \cite{Gor2} considered the problem \eqref{def:hormconstant2} for general $d$, while Bergman \cite{B} considered the one-dimensional question in the setting of de Branges spaces.

In the spirit of \cite{BORS2}, our aim here is to provide a good description of the extremizer of \eqref{def:hormconstant2}, which we denote by $\p_d$ and baptize here as the \textit{Hörmander--Bernhardsson function in dimension $d$}.
This function is of exponential type 1, even, real on the real line, and
has only real simple zeros $\{t_{d,n}\}_{n \neq 0}$, which we order in increasing fashion as
$$
\cdots<t_{d,-2}<t_{d,-1}<0<t_{d,1}<t_{d,2}<\cdots,
$$
with $t_{d,-n}=-t_{d,n}$. As such, the function $\p_d$ admits the Hadamard factorization
$$
    \p_d(z) = \prod_{n = 1}^\infty \bigg( 1- \frac{z^2}{t_{d,n}^2} \bigg),
$$
where we normalize $\p_d(0)=1$. First variation calculations also show that $\p_d$ is uniquely defined by the property
\begin{align}\label{id:integralinterp}
\int_{\r} f(x) \,\sign(\p_d) |x|^{d-1}\d x = {\C_d^{-1}} f(0)
\end{align} 
for all $f\in PW_{d-1}$. 
These properties about $\p_d$ were proved by Dai, Gorbachev, and Tikhonov in \cite{DGT}. Nevertheless, for the sake of completeness we include a brief derivation of these facts in Section \ref{sec:dim-reduction}.

In dimension $d = 1$, the authors of \cite{BORS2} found that after factoring the extremizer as $\p_1(z) = \P(z) \P(-z)$, the function $\P$ can be characterized as the solution to both a functional equation and an ordinary differential equation. A differential equation of $\p$ can be recovered as a symmetric square of the differential equation of $\P$. Gorbachev \cite{Gor2} recently extended the differential equation characterization of the H\"ormander--Bernhardsson function to odd dimensions. We introduce an independent method completing the gap, proving this description holds for all dimensions $d$.

\begin{theorem}\label{thm:ode}
The function $\p_d(z)$ satisfies the third-order differential equation
\begin{align*}
    z^{2d+2}f'''(z) +3(d+1)z^{2d+1}f''(z) +\big((d+1)(2d+1)z^{2d} -r(z) \big)f'(z) - \frac12 r'(z)f(z) = 0,
\end{align*}
where $r(z)$ is the degree $2d + 2$ polynomial given by
\begin{equation*}
    r(z) =  - z^{2d+2} +  \frac{d^2}{4\C_d^2} \sum_{j = 0}^{d } \alpha_{j} z^{2j}
\end{equation*}
and $\alpha_j$ is the coefficient of $z^{2j}$ in $\frac{1}{\p_d(z)^2}$ for  $j = 0, \ldots, d$. Equivalently, $\p_d$ is a solution to 
\begin{equation*}
    \partial_z\big[z^{d+1}\partial_z(z^{d+1}f'(z))\big]
    =\frac{1}{2f(z)} \partial_z\big[r(z) f(z)^2\big]
\end{equation*}
and
\begin{equation*}
    2f(z) z^{d+1} \partial_z( z^{d+1}f'(z) ) - (z^{d+1}f'(z))^2 = r(z) f(z)^2-\frac{d^2}{4\C_d^2}.
\end{equation*}
\end{theorem}

Gorbachev's method in \cite{Gor2} consists of introducing an analytic signature and using a Fourier spectral gap argument, but it relies on the fact that $|x|^{d-1}$ is a polynomial, which only happens when $d$ is odd. Our approach goes through an alternative route, deriving an interpolation formula out of the Euler--Lagrange equation for the variational problem \eqref{def:hormconstant2}. In loose terms, we then apply the interpolation formula to $z^{-d -1}$, suitably modified so that it lies in the space. This allows us to obtain what we call here \textit{reciprocal formulas}, and they suggest the convenient factorization of $\p_d$ as $\p_d (z) = \P_+ (z) \P_-(z)$. These factorizations will be different according to the parity of the dimension, and they coincide with the factorizations of \cite{BORS2,DGT} only when $d$ is odd. Our derivation of the differential equation then follows a path which is at first similar, since our reciprocal formula is equivalent to the quadratic functional identity of \cite[Theorem 1(a)]{Gor2}, but then diverges again, since we use different methods to pin down the polynomial coefficients appearing in the equation. 

It is also worth pointing out that all results in this manuscript were obtained in late 2025, but we had intended to derive the asymptotics of $\C_d$ from our methods before posting. The appearance of Gorbachev's paper \cite{Gor2} prompted us to post now. Nevertheless, we believe our techniques are more natural and direct.

\section{Preliminaries}\label{sec:dim-reduction}

We discuss the dimension reduction of the extremal problem and a few basic properties of the H\"ormander--Bernhardsson function. Most of the properties here were proven elsewhere (e.g., in \cite{BCOS,DGT,Gor2,LL}), but we include this brief discussion to make the paper self-contained.

\subsection{Dimension Reduction}
In the introduction, we asserted that we could reduce the point evaluation problem in several dimensions to a one-dimensional problem. The idea is to show that radial extremizers exist, leading to equation \eqref{def:hormconstant2}. 
First, we observe that
\begin{align}\label{def:hormconstant}
    \frac{1}{\C_d}= \frac2{|\sp^{d-1}|} \inf_{F\in PW^1(\r^d) \, : \,  F(0)=1} \|F\|_{1}.
\end{align}
Indeed, let $c^{-1}$ denote the quantity on the right hand side above. It is clear that if $G(x)=F(x+b)/F(b)$, where $|F(b)|=\|F\|_\infty$, then
$G(0)=1$ and $c\geq \frac{|\sp^{d-1}|}{2}  \frac1{\|G\|_1} = \frac{|\sp^{d-1}|}{2} \frac{\|F\|_{\infty}}{\|F\|_{1}}$. Thus, $c\geq \C_d$. On the other hand, given $G\in PW^1(\r^d)$
with $G(0)=1$ we have $\C_d \geq \frac{|\sp^{d-1}|}{2} \frac{\|G\|_{\infty}}{\|G\|_{1}} \geq \frac{|\sp^{d-1}|}{2} \frac{1}{\|G\|_{1}}$, hence $\C_d\geq c$.

A standard argument shows that extremizers of \eqref{def:hormconstant} must exist. Next, we set out to prove such extremizers can be taken real and radial. If $F$ is an extremizer, it must be real, otherwise taking $\frac12 (F(z) + \overline{F(\bar z)})$ would produce a strictly smaller $L^1$-norm. This latter function is real entire and belongs to  $PW^1(\r^d)$. 
First variation shows that 
$$
\int_{\r^d} \sign(F) G(x) \d x = 0
$$
for all $G\in PW^1(\r^d)$ with $G(0)=0$. This shows that $\sign(F)$ cannot be constant. We claim \eqref{def:hormconstant} admits radial extremizers. Indeed, let
$$
F_0(x) = \int_{O(d)} F(\rho x) \d \mu(\rho),
$$
where $\mu$ is the Haar measure on the orthogonal group $O(d)$. We obtain that $F_0(0)=1$ and $\|F_0\|_1 \leq \|F\|_1$. 
Hence $\|F_0\|_1 = \|F\|_1$ and $F_0$ is a radial extremizer.

It is routine to show that a radial function $F\in PW^1(\r^d)$ can be uniquely written as
$F(x)=f(|x|)$ for some even function $f\in PW^1(\r)$ such that 
$$
\|f\|_{|\cdot|^{d-1}}=\int_\r |f(x)||x|^{d-1}\d x = \frac{2}{|\sp^{d-1}|} \int_{\r^d} |F(x)|\d x <\infty.
$$
We will write, for short notation, $PW_{d-1}$ as the space of functions $f:\r\to\cp$ such that $\text{supp}(\ft f)\subset [-1,1]$ and $\int_\r |f(x)||x|^{d-1}\d x<\infty$. In particular, we conclude that
\begin{align}
\frac{1}{\C_d}=  \inf_{f\in PW_{d-1} \, : \,  f(0)=1} \|f\|_{|\cdot|^{d-1}},
\end{align}
which is exactly the assertion \eqref{def:hormconstant2} made in the introduction. Not only that, radial extremizers of \eqref{def:hormconstant} and \eqref{def:hormconstant2} are mapped into each other by the relation 
$$
F(x)=f(|x|).
$$

\subsection{Basic properties of the H\"ormander--Bernhardsson function}We now show that there is a unique extremizer of \eqref{def:hormconstant2}. Moreover, this function has exponential type 1, is even, real entire, and has real simple zeros. 

Let $f$ be any extremizer of  \eqref{def:hormconstant2}. First we note that $f$ cannot have exponential type $r < 1$, since otherwise the scaled version $f(z/r)$ would still be admissible for \eqref{def:hormconstant2}, while having a strictly smaller integral than $f$, a contradiction.
As before, considering $\frac12 (f(z) + \overline{f(\bar z)})$ leads us to conclude that $f$ must be real entire.

First variation again shows that 
\begin{align}\label{id:firstvar}
    \int_{\r} \sign(f) g(x) |x|^{d-1}\d x = \C_d^{-1} g(0)
\end{align}
for every $g\in PW_{d-1}$. Any function~$f$ satisfying the above property (together with $f(0)=1$) is a minimizer of \eqref{def:hormconstant2}. Indeed, if $f_1$ satisfies the above condition then
\begin{align*}
\int_\r |f(x)||x|^{d-1}\d x & \geq \int_\r f(x)\, \sign(f_1)|x|^{d-1}\d x  \\
&= \int_\r (f(x) -f_1(x))\,\sign(f_1)|x|^{d-1}\d x + \int_\r |f_1(x)||x|^{d-1}\d x \\
& = \int_\r |f_1(x)||x|^{d-1}\d x,
\end{align*}
because $\int_\r (f(x) -f_1(x))\sign(f_1)|x|^{d-1}\d x = \C_d^{-1} (f(0)-f_1(0))=0$. If $f_1,f_2$ are minimizers, then $(f_1+f_2)/2$ must also be an extremizer, in particular we must have that $\sign(f_1)=\sign(f_2)$ (otherwise $(f_1+f_2)/2$ would have smaller $ \|\cdot\|_{|\cdot|^{d-1}}$-norm). We conclude all minimizers have the same sign changes, hence these sign changes are uniquely defined.

Not only that, the extremizer itself is unique. Indeed, picking any other extremizer $f_1$, the condition $\sign f = \sign f_1$ implies $f \cdot f_1 \geq 0$, so by Krein's factorization theorem $f(z) f_1(z) = h(z) \overline{h(\bar z)}$ for a function $h$ of exponential type at most $1$. One can then take $h(z)/h(0)$, and apply the AM--GM inequality to get
$$
    \int_\r |h(x)/h(0)| |x|^{d-1} \d x =  \int_\r \sqrt{|f(x)f_1(x)|} |x|^{d-1}\d x \leq \frac12 \int (|f(x)| + |f_1(x)|) |x|^{d-1} \d x = \frac{1}{\C_d},
$$
since $|h(0)|^2 = f(0) f_1(0) =1$. But $h(z)/h(0)\in PW_{d-1}$ and $\C_d^{-1}$ is a minimum, so the equality  $|f| \equiv |f_1|$ must hold almost everywhere. The sign condition and continuity then allow us to conclude $f \equiv f_1.$ 

From uniqueness, we can conclude $f$ is even and has only real simple zeros. Since $g(z) = {(f(z) + f(-z))}/{2}$ is also an extremizer, it follows $f$ must be even. Now, if $z=a$ is a real zero of order $2k$ then $f(z)/(1-z/a)^{2k}$ has the same sign changes as $f$, so it must be an extremizer. We perform a similar trick if $z=a$ is a real zero of order $2k+1$ with $k \geq 1$, dividing $f(z)$ by $(1-z/a)^{2k}$ so that a simple zero remains at $z =a$. Finally, if $z=a$ is a complex zero of order $k$, then so is $z=\bar a$ since $f$ is real entire, so that $f(z)(1-z/a)^{ -k}(1-z/{\ov a})^{-k}$ has the same sign changes as $f$ and must be an extremizer. Since the result of these procedures is always an extremizer of \eqref{def:hormconstant2}, this would contradict uniqueness, and so it follows $f$ has only real simple zeros. 

\section{Summation and Interpolation Formulas}

Here we lay the groundwork: the summation and interpolation formulas which will allow us to eventually obtain the differential equation of Theorem \ref{thm:ode}. From now on, we shall drop all dependencies on $d$ when convenient and we shall use the notation
$$
c_d:= \frac{d}{4\C_d}. 
$$ 

\subsection{Summation formula for odd functions}
Let $f\in PW_{d-1}$ be odd. Applying \eqref{id:firstvar} and integration by parts, we obtain
\begin{align*}
    \C_d^{-1}f'(0) & = 2\int_0^\infty f'(x)\, \sign(\p(x))|x|^{d-1} \d x \\
    & = 4 \sum_{n\geq 1}(-1)^{n+1} t_n^{d-1}f(t_n) -  2(d-1)\int_0^\infty x^{-1}f(x)\, \sign(\p(x))|x|^{d-1} \d x \\
    & = 4 \sum_{n\geq 1}(-1)^{n+1} t_n^{d-1}f(t_n) -(d-1)\C_d^{-1}f'(0).
\end{align*}
We proved the following proposition.

\begin{proposition}\label{prop:oddsumm}
For every odd $f\in PW_{d-1}$ we have
$$
\sum_{n\geq 1} (-1)^{n+1}t_n^{d-1} f(t_n) = c_d f'(0).
$$
\end{proposition}

\subsection{Interpolation} Out of the summation formula, we seek to prove an interpolation result involving the zeros of the extremal function. 

Define $\p_0(z):=z\p(z)$. Consider the odd function $f(z)=\p_0(z)/(z^2-t_n^2)$ and apply Proposition \ref{prop:oddsumm} to get
$$
c_d(-1/t_n^2) = (-1)^{n+1}t_n^{d-1}\p_0'(t_n)/(2t_n).
$$
This is
\begin{align}\label{id:derivphi}
\p'(t_n)=2c_d (-1)^n/t_n^{d+1}, \qquad n \geq 1.
\end{align}
To treat $n \leq -1$ one can simply use $\p'(-t_n) = -\p'(t_n)$, which holds by the evenness of $\p$. 

Now let $f\in PW_{d-1}$ be an even function and consider the following odd function of $w$
$$
w\mapsto g_z(w)=\frac{zf(z)\p_0(w)-wf(w)\p_0(z)}{w^2-z^2} \in PW_{d-1}.
$$
We can then apply Proposition \ref{prop:oddsumm} to get
$$
\sum_{n\geq 1} (-1)^{n}t_n^{d}\frac{f(t_n)\p_0(z)}{t_n^2-z^2} = c_d \frac{zf(z)-f(0)\p_0(z)}{-z^2}.
$$
Rearranging terms, we have just shown the following proposition.

\begin{proposition}\label{prop:eveninterp}
For every even $f\in PW_{d-1}$, we have
$$
\frac{f(z)}{\p(z)} = f(0) +\frac{z^2}{c_d}\sum_{n\geq 1} \frac{(-1)^{n}t_n^{d}f(t_n) }{z^2-t_n^2}.
$$
\end{proposition}



Note that if we put $t_0=0$ then the above formula is simply Lagrange interpolation
$$
f(z)=\sum_{n\in \z} f(t_n)\frac{\p_0(z)}{\p_0'(t_n)(z-t_n)},
$$
where the sum is taken symmetrically. Given a generic function $f\in PW_{d-1}$, one can split it into an even and odd part, apply the above formula for the even part and also the odd part divided by $z$ and obtain the following result.  
 
\begin{proposition}\label{prop:geninterp}
For every $f\in PW_{d-1}$, we have
$$
\frac{f(z)}{\p(z)}  = {f(0)+f'(0)z}+ \frac{z^2}{2c_d}\sum_{n\geq 1}(-1)^n t_n^{d-1}\bigg[\frac{f(t_n)}{z-t_n}-\frac{f(-t_n)}{z+t_n}\bigg].
$$
\end{proposition}

\begin{remark}
It is important to notice that \emph{all} the results and discussion up to (and including) this section works as it is when $d>0$ is any \emph{real} number. 
\end{remark}

\subsection{Reciprocal formulas: Odd \texorpdfstring{$d$}{d}}
Expand the $1/\p$ in its Taylor series about the origin:
$$
\frac{1}{\p(z)} = \sum_{j\geq 0} \ga_j z^j.
$$
Recall that since $\p$ is even we have $\ga_{2j+1}=0$ for all $j$. 
Now let
$$
f(z)=\frac{1-\p(z)\sum_{j=0}^{d-1} \ga_j z^j}{z^{d+1}}
$$
and observe that $f$ is even and that $f\in PW_{d-1}$. We then apply Proposition \ref{prop:eveninterp} to get
\begin{align}\label{id:sumphi}
\frac{1}{z^{d+1}\p(z)} - \sum_{j=0}^{d-1} \ga_j z^{j-d-1} & = \ga_{d+1} +\frac{z^2}{c_d}\sum_{n\geq 1} \frac{(-1)^{n} }{t_n(z^2-t_n^2)}.
\end{align}
We can put $z=iy$ and let $y\to\infty$ to obtain that
$$
\sum_{n\geq 1} \frac{(-1)^{n+1}}{t_n} = c_d \ga_{d+1},
$$
since by standard alternating series estimates, for any fixed $N$,
$$
    \Bigg|\sum_{n = 1}^\infty(-1)^n \frac{1}{t_n}  - \sum_{n = 1}^\infty (-1)^n \frac{y^2}{t_n(t_n^2 + y^2)}\Bigg| \leq \Bigg|\sum_{n = 1}^N(-1)^n \frac{1}{t_n}  - \sum_{n = 1}^N (-1)^n \frac{y^2}{t_n(t_n^2 + y^2)}\Bigg| + \frac{C}{t_{N+1}}, 
$$
so one can send $y \to \infty$ first and then $N\to \infty$ to conclude the desired limit holds.

Plugging this identity back into identity \eqref{id:sumphi} and using that 
$$
\frac{z^2(-1)^{n} }{t_n(z^2-t_n^2)} = \frac{(-1)^nt_n}{z^2-t_n^2} + \frac{(-1)^{n} }{t_n} =  \frac12\left(\frac{1}{z-(-1)^n t_n} - \frac{1}{z+(-1)^n t_n}\right)+\frac{(-1)^{n} }{t_n}
$$
we obtain the following result.
\begin{proposition}\label{prop:reciprocalformulaodd}
For odd $d$ we have
\begin{align*}
\frac{2c_d}{z^{d+1}\p(z)} & = 2 c_d \sum_{j=0}^{d} \ga_j z^{j-d-1}  + 2\sum_{n\geq 1} \frac{(-1)^nt_n}{z^2-t_n^2}  \\
& = 2 c_d \sum_{j=0}^{d} \ga_j z^{j-d-1} + \sum_{n\geq 1} \frac{1}{z-(-1)^n t_n} -  \sum_{n\geq 1} \frac{1}{z+(-1)^n t_n}.
\end{align*}
\end{proposition}
We have put the polynomial sum up to $j=d$ for convenience, but remember that $\ga_{d}=0$ since $d$ is odd.

 \subsection{Reciprocal formula: Even \texorpdfstring{$d$}{d}} 
We start similarly with the function
$$
f(z)=\frac{1-\p(z)\sum_{j=0}^{d} \ga_j z^j}{z^{d+2}}
$$
Proposition \ref{prop:oddsumm} applied to $z f(z)$ and Proposition \ref{prop:eveninterp} applied to $f(z)$ respectively imply that
$$
\sum_{n\geq 1} \frac{(-1)^{n+1}}{t_n^2} = c_d \ga_{d+2}
$$
and
\begin{align*}
\frac{1}{z^{d+2}\p(z)} - \sum_{j=0}^{d} \ga_j z^{j-d-2} & = \ga_{d+2} +\frac{z^2}{c_d}\sum_{n\geq 1} \frac{(-1)^{n}}{t_n^2(z^2-t_n^2)} = \frac{1}{c_d}\sum_{n\geq 1} \frac{(-1)^n}{z^2- t_n^2}.
\end{align*}
Multiplying both sides by $2zc_d$ and separating even numbered zeros from odd ones, we obtain.

\begin{proposition}\label{prop:reciprocalformulaeven}
For even $d$ we have
\begin{align*}
\frac{2c_d}{z^{d+1}\p(z)} & = 2 c_d \sum_{j=0}^{d} \ga_j z^{j-d-1} +2z \sum_{n\geq 1} \frac{(-1)^n}{z^2- t_n^2} \\ 
& = 2 c_d \sum_{j=0}^{d} \ga_j z^{j-d-1} +\sum_{k\geq 1}\bigg[ \frac{1}{z-t_{2k}}+\frac{1}{z+t_{2k}}\bigg] - \sum_{k\geq 1}\bigg[ \frac{1}{z-t_{2k-1}}+\frac{1}{z+t_{2k-1}}\bigg] .
\end{align*}
\end{proposition}

\section{The Differential Equation: Proof of Theorem 1}
Recall that we have defined $\frac{1}{\p(z)} = \sum_{j\geq 0} \ga_j z^j$ and $c_d:= \frac{d}{4\C_d}$. We start with the Taylor truncation
$$
    p(z)=\sum_{j=0}^{d} \ga_j z^{j}.
$$
Now define for odd $d$ the pair of functions
$$
\P_+(z) := \prod_{n\geq 1} (1-(-1)^nz/t_n) \quad \text{and} \quad \P_-(z) :=\prod_{n\geq 1} (1+(-1)^nz/t_n) \ \ (=\Phi_+(-z)).
$$
For even $d$ define rather
$$
\P_+(z) := \prod_{n\geq 1} (1-z^2/t_{2n}^2) \quad \text{and} \quad \P_-(z) :=\prod_{n\geq 1}  (1-z^2/t_{2n-1}^2).
$$
These definitions are purely inspired by the reciprocal formulas on the previous section. In either case, we have $\p(z) = \P_+ (z)\P_-(z)$ provided these are well defined, which we proceed to justify now. Call $Z_\ep$ the zero set of $\P_\ep$ for $\ep \in \{-,+ \}$. The key observations are that
$\sum_{\tau \in Z_\ep} \frac{1}{\tau^2} < \infty$ and $\sum_{\tau \in Z_\ep\cap [-r, r]} \frac{1}{\tau} \ \text{ converges as } r\to\infty$.
The first claim follows because each $Z_\ep$ is contained in the zero set of the exponential type function $\p$ which implies $\#(Z_\ep \cap [-r, r])  \lesssim r$, while the second follows by the convergence of alternating series when $d$ is odd and exact cancellation when $d$ is even. Thus the products defining $\P_+$ and $\P_-$ converge uniformly in compact sets, and these will be entire functions of exponential type by a theorem of Lindelöf \cite[p. 20]{K}.  

We now prove that the functions $\P_{+}$ and $\P_{-}$ satisfy a second-order differential equation.

\begin{lemma}\label{lem:ode-bigP}
The functions $\P_\ep$ for $\ep \in \{- ,+\}$ satisfy the differential equation
$$
z^{d+1}\P_\ep''(z) + ((d+1)z^d+2\ep c_d p(z))\P'_\ep(z) + \ep c_d p'(z) \P_\ep(z) = q(z) \P_\ep(z),
$$
where $q$ is an entire function of exponential type.
\end{lemma}
\begin{proof}
Observe that, since $\p(z)=\Phi_+(z)\Phi_-(z)$, we have 
\begin{align}\label{id:philogder}
\frac{\p'(z)}{\p(z)}=\frac{\P_+'(z)}{\P_+(z)}+\frac{\P_-'(z)}{\P_-(z)}.
\end{align}
Moreover, Propositions \ref{prop:reciprocalformulaodd} and \ref{prop:reciprocalformulaeven} also imply that
\begin{align}\label{id:recphi}
   \frac{2c_d}{z^{d+1}\p(z)} = \frac{\P_+'(z)}{\P_+(z)}-\frac{\P_-'(z)}{\P_-(z)} +2c_d z^{-d-1}p(z),
\end{align}
which upon rearrangement of terms is
$$
z^{d+1}\P_\ep'(z) = \frac{z^{d+1}\P_\ep(z)\P'_{-\ep}(z) +2\ep c_d}{\P_{-\ep}(z)} -2\ep c_d p(z)\P_\ep(z),
$$
for $\ep=\pm $. Now, if $T$ is a zero of $\P_\ep$, the definition of $\P_\ep$ and identity \eqref{id:derivphi} imply that $T^{d+1} \p'(T) = 2 \ep c_d$, meaning that the function $z^{d+1} \P_\ep'(z)\P_{-\ep}(z)-2\ep c_d$ vanishes at all zeros of $\P_\ep$, since at those points $\p'(T)  = \P_\ep'(T )\P_{-\ep} (T)$. This then implies that
$$
\partial_z [z^{d+1}\P_\ep'(z) + 2 \ep c_d p(z) \P_\ep(z)]_{z=T} = \frac{(T^{d+1}\P_\ep'(T)\P_{-\ep}(T) -2\ep c_d)\P_{-\ep}'(T)}{\P_{-\ep}(T)^2} = 0
$$
if $\P_\ep(T)=0$. We conclude that
$$
z^{d+1}\P_\ep''(z) + ((d+1)z^d+2\ep c_d p(z))\P'_\ep(z) + 2\ep c_d p'(z) \P_\ep(z) = q_\ep(z) \P_\ep(z)
$$
for some entire function $q_\ep$. Taking the difference of the equation with $\ep=+$ and $\ep=-$ (after dividing each by the corresponding $\P_\ep$), and using identities \eqref{id:philogder} and \eqref{id:recphi}, we obtain 
\begin{align*}
q_+(z)-q_-(z)&=
z^{d+1}\partial_z\left[
\frac{2c_d}{z^{d+1}}
\left(\frac{1}{\varphi(z)}-p(z)\right)\right]
+\frac{2c_d(d+1)}{z}\left(\frac{1}{\varphi(z)}-p(z)\right) \\
&\qquad+2c_d\left(\frac{1}{\varphi(z)}-p(z)\right)\frac{\varphi'(z)}{\varphi(z)}
+2c_d p(z)\frac{\varphi'(z)}{\varphi(z)}+4c_dp'(z)\\
&= 2c_d p'(z),
\end{align*}
so one can deduce that the equation in the statement holds with $q(z) := q_+(z) - c_d p'(z)=q_-(z) + c_d p'(z)$. Since $q$ is a ratio of exponential type functions which is entire, it is also of exponential type \cite[p. 22]{K}.
\end{proof}

Actually, we can derive something much stronger about $q$ from the differential equation: it is a polynomial given by the formula
\begin{equation*}
    q(z) = -\frac{1}{4} z^{d+1} + \frac{c_d^2}{z^{d+1}}\Bigg(\sum_{j = 0}^{d } \alpha_{j} z^{2j} -p(z)^2\Bigg),
\end{equation*}
where $\alpha_j$ is the coefficient of $z^{2j}$ in $\frac{1}{\p_d(z)^2}$. This will be a consequence of Lemmas \ref{lem:q_poly} and \ref{lem:r_poly}, which we prove next.

\begin{lemma}\label{lem:q_poly} The function $q(z)$ is a polynomial of degree at most $d+1.$
\end{lemma}
\begin{proof}
The strategy consists of proving a bound for $q$ of the form
\begin{equation*}
    |q(z)| \lesssim|z|^{d+1} + 1
\end{equation*}
via an application of the Phragmén--Lindelöf principle. Once that is in place, the result plainly follows from a standard application of Liouville's theorem and Cauchy's integral formula.

Adding \eqref{id:philogder} and \eqref{id:recphi} we have
\begin{equation}\label{id:logDerPhiPlus}
    \frac{\Phi_+'(z)}{\P_+(z)}  = \frac{1}{2} \frac{\varphi'(z)}{\varphi(z)} + \frac{ c_d}{z^{d+1} \varphi(z)} - \frac{c_d p(z)}{z^{d+1}}.
\end{equation}
Let $t>0$ and $M \in \r$ and plug $z =t + i tM $ in the above formula. We obtain that
\begin{equation*}
    \frac{\Phi_+'(t + itM )}{\P_+(t + itM)}  = \frac{1}{2} \frac{\varphi'(t + itM )}{\varphi(t + itM )} + O(1),
\end{equation*}
where the $O(1)$ term is uniform in $|t| \geq (M^2 + 1)^{-1/2}$, since $\deg p \leq d$ and $|\varphi(t + it M)| \geq 1$ for all $t > 0$ and $M > 1$ (this can be seen from the product formula for $\p$).
We also have
\begin{align*}
    \bigg|\frac{\varphi'(t + itM)}{\varphi(t + itM)} \bigg|&=  2 \bigg|\sum_{n\geq 1} \frac{t + itM}{(t + itM)^2 - t_n^2}\bigg| \\
    &= 2 \bigg|\sum_{n\geq 1} \frac{t + itM}{(t^2(1- M^2) +2it^2M )- t_n^2}\bigg|\\
    &\leq 2(t^2 + t^2 M^2 )^{1/2} \sum_{n\geq 1} \frac{1 }{((t^2( M^2 -1) + t_n^2 )^2+ 4t^4 M^2)^{1/2}}\\
    &\lesssim |Mt| \sum_{n\geq 1} \frac{1 }{((t^2 M^2 + t_n^2 )^2+ 4t^4 M^2)^{1/2}}\\
    &\lesssim   \sum_{n\geq 1} \frac{|Mt| }{t^2 M^2 + t_n^2 } 
\end{align*}
where the implicit constants are universal for $t > 0$ and $|M| > 2$. Now, since $\p$ has exponential type, we have the estimate for the zero counting function of $\p$ given by $n(r) =  \# \{ k  \geq 1 : t_k \leq r \} \lesssim r$ so we can further bound integrating by parts
\begin{equation*}
    \bigg|\frac{\varphi'(t + itM)}{\varphi(t + itM)} \bigg| \lesssim \sum_{n\geq 1} \frac{|Mt| }{t^2 M^2 + t_n^2 }  =  \int_{0}^\infty \frac{|Mt| }{t^2 M^2 + x^2 } \d n(x)\lesssim \int_0^\infty \frac{|Mt|}{|Mt|^2 + x^2} \, \d x = \frac{\pi}{2}. 
\end{equation*}
So one can say that for $M > 2$, the estimate $|\P_+'/\P_+ (t +itM)| \lesssim 1$ holds for all $|t| \geq (M^2 + 1)^{-1/2}.$ The same calculation applies to $\ep = -$.

To obtain the desired bound for $\P''_\ep/ \P_\ep$, one needs only to differentiate :
\begin{align*}
    \frac{\P_\ep''}{\P_\ep}(z)
    = \bigg(\frac{\P_\ep'}{\P_\ep}\bigg)'(z)+ \bigg(\frac{\P_\ep'}{\P_\ep}\bigg)^2(z) = \bigg(\frac{\P_\ep'}{\P_\ep}\bigg)'(z) + O(1).
\end{align*}
Doing similar estimates as before, we thus have
\begin{equation*}
    \frac{\P_\ep''}{\P_\ep}(z)= \frac{1}{2}  \left(\frac{\p'}{\p}\right)' (z) +O(1)
\end{equation*}
whenever $z$ is of the form $z = t + itM$ for $|M| > 2$ and $|t| \geq (M^2 + 1)^{-1/2}$.  To bound the remaining term, observe that
\begin{align*}
    \Bigg|\left(\frac{\p'}{\p}\right)' (t + itM)\Bigg| &= 2 \Bigg|\sum_{n \geq 1} \frac{(t+ itM)^2 + t_n^2}{((t+ itM)^2 - t_n^2)^2} \Bigg|
    \lesssim \sum_{n\geq 1} \frac{(M^2 + 1)t^2 + t_n^2  }{(t^2 M^2 + t_n^2 )^2+ 4t^4 M^2}  \lesssim 1,
\end{align*}
for all $ t > 0 $ and $M > 2$. 

Now let $M > M_0$, so that by the bounds we have obtained 
\begin{align*}
     |q(z)| = \left|\frac{z^{d+1}\P_\ep''(z) + ((d+1)z^d+2\ep c_d p(z))\P'_\ep(z) + \ep c_d p'(z) \P_\ep(z) }{\P_\ep (z)} \right|\lesssim  |z|^{d+1} + 1
\end{align*}
whenever $z = t + itM$ and $|t| \geq (M^2 + 1)^{-1/2}$. Since $q$ is entire, we can trivially estimate $|q(z)| \lesssim 1$ for $|z | \leq 1$. Combining these two facts, it follows that there is a universal constant $C$ such that for every $|M| > M_0$ and every half-line $\ell_M := \{t + itM : t \geq 0\}$ , the bound 
$$|q(z)/(z+1)^{d+1}| \leq C$$
holds for $ z \in \ell_M.$ Given an $M > M_0$, we now apply the Phragmén--Lindelöf principle to $q(z)/(z+1)^{d+1}$ in the sector $\S_M = \{ z : |\arg z| \leq \arctan(M) \}$. We can do this since the central angle $2 \cdot\arctan M$ is strictly less than $\pi$ and $q(z)$ has order at most $1$. Thus, we can use the same constant $C$ to bound $|q(z)/(z+1)^{d+1}| \leq C$ for all $z \in \S_M$. Since $M > M_0$ can be taken arbitrarily large, it follows $|q(z)/(z+1)^{d+1}| \leq C$ in the right half-plane $\{x+iy: x > 0 \}$. The same argument can be applied to $|q(z)/(z-1)^{d+1}|$ to bound it by (a possibly larger) $C$ in the left half-plane $\{x+iy: x < 0 \}$. Hence,
\begin{equation*}
    |q(z)| \lesssim |z|^{d+1} + 1
\end{equation*}
holds in all of $\cp$ by continuity.
\end{proof}
\vspace{-0.25cm}
Next, we will consider an auxiliary differential equation for which the ODE in Theorem \ref{thm:ode} will be its symmetric square. Define 
\begin{equation*}
    h_d(z) = \sum_{k = 0}^{d -1} \frac{\ga_k}{k - d} z^{k-d} + \ga_d \log z
\end{equation*}
in a simply connected domain avoiding the origin. Define
\begin{equation*}
    \Psi_\ep (z) = e^{\ep c_d h_d(z)} \P_\ep(z),
\end{equation*}
so that from Lemma \ref{lem:ode-bigP}, it satisfies the differential equation
\begin{equation*}
    \Psi_\ep ''(z) + \frac{d+1}{z} \Psi_\ep'(z) - \frac{z^{d+1}q(z) +  c_d^2 p(z)^2}{z^{2d+2}}  \Psi_\ep(z) = 0.
\end{equation*}
Therefore, since $\Psi_+$ and $\Psi_-$ solve the same second-order ODE, the product $\p(z) = \Psi_+(z) \Psi_-(z)$ is a solution of the symmetric square of the above differential equation, yielding
\begin{equation*}
    \p'''(z) + \frac{3(d+1)}{z} \p''(z) + \bigg(\frac{(d+1)(2d+1)}{z^2} -\frac{r(z)}{z^{2d+2}} \bigg)\p'(z) - \frac{r'(z)}{2z^{2d+2}}\p(z) = 0,
\end{equation*}
where $r(z)$ is the polynomial given by
\begin{equation}\label{def:r}
    r(z) := 4z^{d+1}q(z)+4c_d^2p(z)^2.
\end{equation}
Upon multiplying by $z^{2d+2}$, this becomes
\begin{align*}
    z^{2d+2}\p'''(z) +3(d+1)z^{2d+1}\p''(z) +\big((d+1)(2d+1)z^{2d} -r(z) \big)\p'(z) - \frac12 r'(z)\p(z) = 0.
\end{align*}
Using
\begin{equation*}
    \partial_z\big[z^{d+1}\partial_z(z^{d+1}\p')\big]
    =z^{2d+2}\p'''+3(d+1)z^{2d+1}\p''+(d+1)(2d+1)z^{2d}\p',
\end{equation*}
we can rearrange it to obtain 
\begin{align*}
    \partial_z\big[z^{d+1}\partial_z(z^{d+1}\p'(z))\big]
    &=4(z^{d+1}q(z)+c_d^2p(z)^2)\p'(z)+\partial_z\big[2z^{d+1}q(z)+2c_d^2p(z)^2\big]\p(z) \\
    &= r(z) \p'(z) + \frac12 r'(z) \p(z).
\end{align*}

Now, setting $Q(z) := z^{d+1} \partial_z[ z^{d+1}  \p'(z)]$, one can see that the above leads to
\begin{align*}
     \partial_z [r(z) \p(z)^2] &= 2 \p(z) Q'(z) =2\partial_z[\p(z)Q(z)] - 2 \p'(z)Q(z) = \partial_z[2\p(z)Q(z)-z^{2d+2} \p'(z)^2]
\end{align*}
so integrating
\begin{equation*}
     2\p(z) z^{d+1} \partial_z( z^{d+1}\p'(z) ) - (z^{d+1}\p'(z))^2 = r(z) \p(z)^2 + C,
\end{equation*}
where $C = -r(0) = -4 c_d^2$. 

Thus, we only need to prove the following lemma to conclude the proof of Theorem \ref{thm:ode}.
\begin{lemma}\label{lem:r_poly}
    The polynomial $r(z)$ defined in \eqref{def:r} is of degree $2d + 2$ and it is given by
\begin{equation*}
    r(z) =  - z^{2d+2} + 4 c_d^2 \sum_{j = 0}^{d } \alpha_{j} z^{2j},
\end{equation*}
where $\alpha_j$ is the coefficient of $z^{2j}$ in $\frac{1}{\p(z)^2}$ for  $j = 0, \ldots, d$.
\end{lemma}
\begin{proof}
Our first observation is that $r(z) = 4z^{d+1}q(z)+4c_d^2p(z)^2$ is even. From the differential equation of Lemma \ref{lem:ode-bigP}, we can conclude $q$ is even when $d$ is odd and odd when $d$ is even. Thus, the product $z^{d+1} q(z)$ is in either case even. Using the fact that $p(z)$ is also even, we can conclude the same holds for $r(z)$. 

Adding and subtracting $4 c_d^2\p'(z)\p(z)^{-2}$, our differential equation becomes
\begin{align*}
    \partial_z\big[z^{d+1}\partial_z(z^{d+1}\p'(z))\big]
    &= (r(z) -4c_d^2\p(z)^{-2})\p'(z) + \bigg(\frac{1}{2}r'(z) +  4 c_d^2 \p'(z) \p(z)^{-3}\bigg) \p(z).
\end{align*}
We now compare the power series expansion of both sides of the equation. From the expression in the left-hand side, we can see that all coefficients of degree at most $2d$ must vanish. Now, considering $r(z) $ is even, we let $\beta_j$ stand for the coefficient of $z^{2j}$ in $r(z)$. We have that by putting the $(2k+1)$-th coefficient of the right-hand side equal to 0, we obtain
\begin{equation}\label{id:coeffEq}
    \sum_{j = 0}^{k} (\beta_j - 4c_d^2\alpha_j)  \frac{\p^{(2k -2j +2)}(0)}{(2k -2j +1)!} +  \sum_{j = 0}^{k} (j+1)(\beta_{j +1} -4c_d^2\alpha_{j+1})\frac{\p^{(2k-2j)}(0)}{(2k -2j)!} = 0 
\end{equation}
for $k =0, \ldots, d -1$. Taking $k = 0$, this yields $(\beta_0 -4 c_d^2\alpha_0)\p''(0) +  (\beta_1 - 4 c_d^2\alpha_1)\p(0) = 0$. But we already have $\beta_0 = 4c_d^2\alpha_0$ since $\alpha_0 = 1= p(0)^2 $, so $\beta_1 = 4 c_d^ 2\alpha_1$ as well. Proceeding inductively, if we suppose that $\beta_j = 4 c_d^2 \alpha_j$ for $j = 0, \ldots, k$, we have that the equation \eqref{id:coeffEq} reduces to $(k+1)(\beta_{k+1} - 4 c_d^2\alpha_{k+1}) \p(0) = 0$, which gives equality for the case $k+1$, whence we conclude that $\beta_j = 4 c_d^2 \alpha_j$ for all $j = 0, \ldots, d$. 

To confirm $\beta_{d+1} = -1$, we proceed differently. Note that
\begin{equation*}
    \beta_{d+1}=4  \lim_{y\to \infty} \frac{q(iy)}{(iy)^{d+1}} =\lim_{y \to \infty} \left(\frac{\p'(iy)}{\p(iy)}\right)^2,
\end{equation*}
where the middle equality follows from 
\begin{equation*}
    \frac{\P'_\ep (i y)}{\P_\ep(i y)}  = \frac{1}{2}\frac{\p'(iy)}{\p(iy)} + o(1) \quad \text{ and } \quad \bigg(\frac{\P'_\ep}{\P_\ep}\bigg)'(iy) = o(1)
\end{equation*}
as $y\to\infty$, as one can see through calculations similar to those of Lemma \ref{lem:q_poly}. So we proceed to calculate the limit of $\frac{\p'(iy)}{\p(iy)}$ as $y\to \infty$. Letting 
\begin{equation*}
    n(r) := \# \{ k \geq 1  : t_k \leq r \}
\end{equation*}
denote the zero counting function for $\p$, we can write
\begin{equation*}
    \frac{\p'(iy)}{\p(iy)} = -2i \sum_{n\geq 1}\frac{y}{y^2 + t_n^2} = - 2i \int_0^\infty \frac{y}{y^2 + t^2} \, \d n (t)
\end{equation*}
and integrate by parts, yielding
\begin{align*}
    \int_0^\infty \frac{2y}{y^2 + t^2} \, \d n (t) =  \int_0^\infty \frac{4 yt}{(y^2 + t^2)^2} n(t) \, \d t .
\end{align*}
Now, since $\p$ belongs to the Cartwright class ($\p$ is bounded on $\r$ and has exponential type 1), we have precise asymptotics for $n(r)$, namely
\begin{equation*}
    n(r) = \frac{1}{\pi}r + o (r)
\end{equation*}
as $r \to \infty$ (see \cite[p. 127]{Lev}). We can therefore compute 
\begin{align*}
    \int_0^\infty \frac{4 yt}{(y^2 + t^2)^2} n(t) \, \d t = \frac{4}{\pi} \int_0^\infty \frac{yt^2}{(y^2 + t^2)^2} \, \d t + o(1) = 1 + o(1),
\end{align*}
which shows
\begin{equation*}
    \lim_{y \to \infty} \frac{\p'(iy)}{\p(iy)} = -i,
\end{equation*}
and so $\beta_{d+1} = (-i)^2 = -1$, concluding the proof.
\end{proof}

\section{Appendix: a distributional identity}
For convenience we let
$$A_d(z)= z^{1-d/2} J_{d/2-1}(z) \ \text{ and } \ B_d(z)= z^{1-d/2}J_{d/2}(z) $$
where $J_\nu(z)$ is the Bessel function of the first kind. They satisfy the following system of equations
$$
A_d'(z)=-B_d(z) \ \ \text{ and } \ \ B_d'(z)=A_d(z)-(d-1)B_d(z)/z.
$$
Observe that $A_d$ is even, $B_d$ is odd, $A_d(0)=[(d/2-1)!2^{d/2-1}]^{-1}$ and $B'_d(0)=[(d/2)!2^{d/2}]^{-1}$. 

Let $E_d(z)=A_d(z)-iB_d(z)$. We define the \textit{generalized Hankel transform} of a function~$f$ as
$$
H_d f(\xi) = \frac12 \int_\r f(x)E_d(x\xi)|x|^{d-1}\d x.
$$
Note that if $d=1$, we are back to the Fourier transform since $E_{1}(z)=(\pi/2)^{-1/2}e^{-iz}$. We note that if $f$ is even then
$$
\xi^{(d-1)/2}H_d f(\xi) = \int_0^\infty [x^{(d-1)/2}f(x)]J_{d/2-1}(x\xi)\sqrt{x\xi}\d x,
$$
and if $f$ is odd, then we have
$$
\xi^{(d-1)/2}H_d f(\xi) = -i\int_0^\infty [x^{(d-1)/2}f(x)]J_{d/2}(x\xi)\sqrt{x\xi}\d x
$$
and these are usually defined Hankel transforms of the function $x^{(d-1)/2}f(x)$ \cite{W}. However, coupled in this way we have that $H_d$ is unitary in $L^2(\r, |x|^{d-1} \d x)$ with inverse $H_d^{-1}f(x)=H_d f(-x)$.
Moreover, $H_d(\S(\r))=\S(\r)$. Furthermore, if $F:\r^d\to\cp$ is radial, then $\ft F(\xi) = H_{d}F(|\xi|)$. 

In various cases, the Hankel transform generalizes the role of the Fourier transform in higher dimensions. One instance of this is the result \cite[Lemma 2.3]{BORS2}, which states that when $d = 1$, the equation
\begin{equation*}
    \sum_{n = 1}^\infty (-1)^{n +1} \sin(t_n \xi) = \frac{\C_1^{-1}}{4 } \xi
\end{equation*}
holds in a distributional sense. For general $d$, one can show the following. 
\vspace{-0.5cm}
\begin{proposition}\label{prop:distid}
 For $|\xi|<1$ we have
$$
\lim_{N\to\infty} \sum_{n= 1}^{N} (-1)^{n+1} B_d( \xi t_n)t_n^{d-1} =   \frac{\C_d^{-1}A_d(0)}{4}  \xi
$$
in the following distributional sense: If $g \in C^\infty(\r)$ is supported in $(-1,1)$ then
$$
\lim_{N\to\infty} \sum_{n= 1}^{N}  (-1)^{n+1} t_n^{d-1} \int_{-1}^1 B_d( \xi t_n) g(\xi) |\xi|^{d-1} \d \xi = \frac{\C_d^{-1}A_d(0)}{4}  \int_{-1}^1 \xi g(\xi) |\xi|^{d-1} \d \xi.
$$
\end{proposition}
\begin{proof}
Since both distributions involved are odd, it is enough to prove the identity for $g\in C^\infty(\r)$ odd and supported in $(-1,1)$. Then letting $h(\xi) = \xi g(\xi)$, by \eqref{id:integralinterp} we have

\begin{equation*}
\int_\r  H_d h(x)\sign(\p(x))|x|^{d-1}\d x ={\C_d^{-1}A_d(0)}\int_{0}^1 \xi g(\xi)|\xi|^{d-1}\d \xi.
\end{equation*}
Writing $H_d h(x) =  \int_{0}^1 A_d(x \xi) \xi g(\xi) |\xi|^{d-1} \d \xi$ and switching integrals we get
$$
\lim_{R\to\infty} \int_{0}^1 \xi g(\xi) \bigg[\int_0^R  A_d(\xi x)\sign(\p(x))x^{d-1}\d x \bigg] |\xi|^{d-1} \d \xi = \frac{\C_d^{-1}A_d(0)}{2 } \int_{0}^1 \xi g(\xi)|\xi|^{d-1}\d \xi.
$$
Now,  using that $[B_d(x)x^{d-1}]' = A_d(x)x^{d-1}$ and that $B_d$ is odd (so $B_d(0)=0$) we obtain
\begin{align*}
\int_0^R  A_d(\xi x)\sign(\p_d(x))x^{d-1}\d x &= \sum_{n=1}^{N}(-1)^{n+1} \int_{t_{n-1}}^{t_n} A_d(\xi x) x^{d-1}\d x + (-1)^N \int_{t_N}^R A_d(\xi x) x^{d-1}\d x\\
& = 2\xi^{-1}\sum_{n= 1}^{N} (-1)^{n+1} B_d( \xi t_n)t_n^{d-1}+ (-1)^N  \xi^{-1}B_d(\xi R) R^{d-1},
\end{align*}
where $N = \max\{n: t_n \leq R\}$, for $|\xi|<1$. However, $\xi^{-1}B_d(\xi R) R^{d-1} = o(1)$ when paired with $\xi g(\xi)$, since this is simply equal to $iR^{d-1}H_d g(R)$, which has fast decay. This shows that 
$$
\lim_{N\to\infty} \sum_{n= 1}^{N} (-1)^{n+1} B_d( \xi t_n)t_n^{d-1} =\frac{\C_d^{-1}A_d(0)}{4} \xi
$$
in distributional sense for $|\xi|<1$. 
\end{proof}

One can use Proposition \ref{prop:distid} and Hankel inversion to obtain an alternative way to deduce Proposition \ref{prop:oddsumm}: for odd $f\in PW_{d-1}$,
\begin{align*}
\lim_N \sum_{n=1}^N (-1)^{n+1}t_n^{d-1} f(t_n) & = \lim_N i \int_{0}^1 H_d f(\xi) \bigg[ \sum_{n=1}^N(-1)^{n+1}t_n^{d-1}B_d(\xi t_n)\bigg] |\xi|^{d-1}\d \xi  \\
& =\frac{\C_d^{-1}A_d(0)}{4 }  i \int_{0}^1 H_d f(\xi) \xi |\xi|^{d-1}\d \xi \\
& = \frac{\C_d^{-1}A_d(0)}{4B'_d(0) }  f'(0) = \frac{d\C_d^{-1}}{4} f'(0).
\end{align*}
Since the zeros satisfy $|t_{n+1} - t_n| \gtrsim 1$ (see \cite[Theorem 1.7]{B}), by the Plancherel--Pólya inequality \cite{PP1,PP2} we conclude that every $f\in PW_{d-1}$ satisfies  $\sum_{n\geq 1}^\infty t_n^{d-1}| f(t_n)|  <\infty$, and so a routine approximation argument proves the desired result. 

\vspace{-0.2cm}
\section*{AI Statement}
Generative AI was used to identify typos and suggest minor corrections. It also pointed out to the authors the two alternative expressions for the ODEs in Theorem 1.

\section*{Acknowledgments}
FG acknowledges support from the following funding agencies: The Office of Naval Research GRANT14201749 (award number N629092412126), The Serrapilheira Institute (Serra-2211-41824), FAPERJ (E-26/200.209/2023) and CNPq (309910/2023-4). DR acknowledges funding by the European Union (ERC, FourIntExP, 101078782).

\end{document}